\documentclass{amsart}
\usepackage{amssymb}
\newtheorem{thm}{Theorem}[section]
\newtheorem{prop}[thm]{Proposition}
\newtheorem{cor}[thm]{Corollary}
\newtheorem{lem}[thm]{Lemma}

\theoremstyle{definition}
\newtheorem{dfn}[thm]{Definition}
\newtheorem{ex}[thm]{Example}

\theoremstyle{remark}
\newtheorem{rem}[thm]{Remark}

\newcommand{\mR}{\mathbb{R}}

\newcommand{\mN}{\mathbb{N}}
\newcommand{\mC}{\mathbb{C}}

\newcommand{\mT}{\mathbb{T}}

\newcommand{\mRP}{\mathbb{RP}}
\newcommand{\mCP}{\mathbb{CP}}
\newcommand{\mHP}{\mathbb{HP}}
\newcommand{\mOP}{\mathbb{OP}}

\newcommand{\whH}{\widehat{H}}
\newcommand{\whM}{\widehat{M}}
\newcommand{\whU}{\widehat{U}}
\newcommand{\whV}{\widehat{V}}
\newcommand{\whpi}{\widehat{\pi}}

\newcommand{\cH}{\ensuremath{\mathcal{H}}}

\newcommand{\cat}{\ensuremath{\mathrm{cat} }}
\newcommand{\Fix}{\operatorname{Fix}}

\newcommand{\intern}{\operatorname{int}}

\begin{document}

\title{The equivariant LS-category of polar actions}
\author{Steven Hurder}
\address{Steven Hurder, Department of Mathematics, University of Illinois at Chicago, USA}
\email{hurder@uic.edu}
\author{Dirk T\"oben}
\address{Dirk T\"oben, Mathematisches Institut, Universit\"at zu K\"oln, Weyertal 86-90, 50931 K\"oln, Germany}
\email{dtoeben@math.uni-koeln.de}
\date{February 20, 2007}
\thanks{SH was supported in part by NSF grant DMS-0406254.}
\thanks{DT was supported by the Schwerpunktprogramm SPP 1154 of the DFG}


\begin{abstract}
We will provide a lower bound for arbitrary proper actions in terms of the stratification by orbit types, and an upper bound for proper polar actions in terms of the equivariant LS-category of its generalized Weyl group. As an application we reprove a theorem of Singhof that determines the classical Lusternik-Schnirelmann category for $U(n)$ and $SU(n)$.
\end{abstract}
\maketitle
\section{Introduction}
The equivariant Lusternik-Schnirelmann category $\cat_G(M)$ of an action by a Lie group $G$ on a manifold $M$ (see Definition \ref{def:equivcat}) was introduced by Marzantowicz in \cite{marzantowicz} for compact $G$, as a generalization the classical Lusternik-Schnirelmann category of a space \cite{James2}.
He showed that similar to the classical case, $\cat_G(M)$ is a lower bound for the number of critical orbits of a $G$-invariant function on $M$ and that  it has a lower bound in terms of the cuplength of a $G$-cohomology theory. 

 Colman studied the equivariant  category $\cat_G(M)$ for $G$ a finite group in \cite{colman}, and gave an upper bound in terms of the category of the connected components of the singular set for the action. 
 Moreover, her work gives examples    of finite group actions on compact surfaces for which   $\cat_G(M)$ can be made arbitrarily large \cite{colman}, showing the necessity of working with the connected components. 
 Note that for finite group actions, the singular set consists entirely of exceptional points.
 
Ayala, Lasheras and Quintero  \cite{ayala} generalize the Marzantowicz results  to proper group actions, although   finite group actions were still their primary consideration. 

In   this paper,  we will focus   on the equivariant category of proper actions by higher dimensional Lie groups. 
 In sections \ref{sec:catbounds} and \ref{sec:lower} we will introduce a refinement of the stratification by orbit types and provide a lower bound for arbitrary proper actions in terms of its bottom stratum. Section~\ref{sec:polar} defines the class of polar actions from the title, and section~\ref{sec:weyl} introduces  the Weyl group associated to a polar action. Section~\ref{sec:upper} contains our main result; for proper polar actions we will give an upper bound in terms of the equivariant category of its generalized Weyl group of a polar action, thereby reducing the computation to the discrete case. In section \ref{sec:examples} we will use the previous results to calculate the equivariant category of $SU(n)$ and $U(n)$. We then observe that this also determines the classical Lusternik-Schnirelmann category of these spaces, which is a theorem by Singhof. 

This work was begun while the second author was visiting the University of Illinois at Chicago for the academic year 2005-2006, and completed while the first author was a guest at the Institut Henri Poincar\'{e} in Paris. The authors would like to express their thanks for the hospitality shown by both host institutions.

\section{Categorical bounds for proper actions}\label{sec:catbounds}

Let $G$ be a topological group acting on a topological space $M$; in most of our cases this will be a Lie group acting smoothly on a manifold. A homotopy $H:U\times [0,1]\to M$ of an open $G$-invariant set $U\subset M$ is called $G$-{\it equivariant} or just $G$-{\it homotopy} if $gH(x,t)=H(gx,t)$ for any $g\in G, x\in U$ and $t\in [0,1]$. We also write $x_t=H_t(x)=H(x,t)$. The set $U$ is $G\mbox{-}${\it categorical}, if $H_0$ is the identity and if $H_1$ maps $U$  to a single orbit. 

\begin{dfn}\label{def:equivcat}
The {\it equivariant category} $\cat_G(M)$ is the least number of $G$-categorical sets required to cover $M$. If there is no categorical cover of $M$ we set $\cat_G(M)=\infty$. 
\end{dfn}

\begin{rem}
For the trivial group $G=\{e\}$ we recover the classical Lusternik-Schnirelmann category, denoted by $\cat(M)$. If $M$ is empty,   $\cat_G(M) = 0$. 
\end{rem}

Recall that an action is \emph{proper}  if for each compact subset $K \subset M$,  $\{(g,x) \mid gx \in K\}$ is a compact subset of $G \times M$.
For a proper action,   the orbits $Gx$ are closed and embedded submanifolds, hence 
the quotient $M/G$ is a Hausdorff space. Then an obvious lower bound for $\cat_G(M)$ is $\cat(M/G)\leq \cat_G(M)$. 

From now on we assume that $G$ is a Lie group acting properly on a manifold $M$. 

An important tool for studying the equivariant category is provided by the following well-known theorem 
\cite{Bredon1972, duistermaat,palais}. 

\begin{thm}[Tubular Neighborhood Theorem]
Let $G$ act properly on $M$. Then for any orbit $Gx$ there is an invariant neighborhood $U$ and a $G$-equivariant homotopy $H:U\times[0,1]\to M$ with $H_1(U)=Gx$ and $H_t|Gx$ is the inclusion of $Gx$ for all $t\in[0,1]$.
\end{thm}
We will give an outline of the proof in the case that the action is smooth, thereby introducing principles that will be useful later. 
\begin{proof}
Choose  a Riemannian metric on $M$ for which the $G$-action  is isometric.  For $x \in M$, the orbit $Gx \subset M$ is a properly embedded smooth submanifold. 
Let  $\nu (Gx) \to Gx$ be the normal bundle to $Gx$ of $M$, and  let $\nu^r (Gx) \subset  \nu (Gx)$ denote the disk subbundle of vectors of length at most $r$. Then 
there exists $r > 0$ such  that the geodesic exponential map  $\exp : \nu^r(Gx) \to M$   is a diffeomorphism onto a  tubular neighborhood $U$  of the orbit $Gx$. 
Define the    geodesic  retraction $H:U\times[0,1]\to M$ onto $Gx$   by $H_t(\exp(\vec{v})) = H(\exp(\vec{v}), t) = \exp(t \vec{v})$.   It is easy to see that the homotopy $H$ is $G$-equivariant. 
\end{proof}

Now define the slice $S=H_1^{-1}(x)$, where $H:U\times [0,1]\to M$ is as in the proof. Let $G_x = \{g \in G \mid gx = x\}$ denote the stabilizer of $x$.  Then there is a $G$-equivariant diffeomorphism 
$$
U\cong G\times_{G_x}S.
$$
\begin{ex}
Let $G$ be a Lie group acting properly on an Hadamard manifold $X$, and assume the action is polar   (see Definition \ref{def:polar}). Let $K$ be a maximal compact subgroup of $G$. One can show $K=G_x$ for some $x\in X$. In \cite{toeben3} the second author proved that the normal exponential map $\exp^\perp:\nu (Gx) \to X$ of $Gx$ is a $G$-equivariant diffeomorphism. In other words, $X$ is a global tubular neighborhood of $Gx$. Therefore $\cat_G(X)=1$.
\end{ex}

The Tubular Neighborhood Theorem directly implies:
\begin{cor}
Let $G$ be a compact Lie group acting on a compact manifold $M$. Then $\cat_G(M)<\infty$.
\end{cor}

 The Tubular Neighborhood Theorem is generalized by the following theorem that is proven by topological methods.
\begin{thm}[Equivariant Borsuk Theorem]
Let $G$ act properly on $M$. Then any closed invariant subset $A$ has an invariant neighborhood $U$ for which $A$ is a strong $G$-deformation retract, i.e. there is $G$-homotopy $H:U\times[0,1]\to M$ with $H_1(U)=A$ and $H_t|A$ is the inclusion of $A$ into $M$ for all $t\in[0,1]$.
\end{thm}
\begin{proof}
This is stated in more generality as Proposition 3.5. in \cite{ayala}. See \cite{borsuk} for the classical version.
\end{proof}

An orbit $Gx$ of maximal dimension is called {\it regular},  and each    point $y \in Gx$ is said to be regular.   Let $r$ be the dimension of a regular orbit.  The \emph{cohomogeneity of the action} is defined to be the codimension $q$ of a regular orbit. An orbit with dimension less than $r$ is said to be \emph{singular}.  The set of regular (respectively, singular) points is denoted by $R$ (respectively,  $S$).  The union of the regular orbits $R$ forms an open dense connected subset of $M$. 

A regular orbit $Gx$ is said to have \emph{non-trivial holonomy} if there exists $y \in Gx$ arbitrarily close to $x$ such that the orbit $Gy$ is a non-trivial covering of $Gx$; such an orbit is said to be \emph{exceptional}. Let $E$ denote the  union of the exceptional  orbits, and $R_0 = R - E$ the regular orbits without holonomy. 
 The quotient space $B := R_0/G$ is a connected  manifold of dimension $q$, which is compact if $R_0 = M$, and open otherwise. The quotient map $\rho \colon R_0 \to B$ is then a right $G$-fiber bundle.

\begin{thm} Let $G \times M \to M$ be a proper smooth action of a Lie group $G$. If either $S$ or $E$ is non-empty, then 
\begin{equation}\label{eq1a}
\cat_G(M) \leq \cat_G(S)+\cat_G(R)\leq \cat_G(S)+\cat_G(E)+q
\end{equation}
Otherwise, if all orbits are regular and there are no exceptional orbits, then 
\begin{equation}\label{eq1b}
\cat_G(M) \leq q +1
\end{equation}
\end{thm}
\begin{proof}

 Let us first show that 
\begin{equation}\label{eq2}
\cat_G(M) \leq \cat_G(S)+\cat_G(R)
\end{equation}
Assume that $S$ is not empty, then by the Equivariant Borsuk Theorem we can extend a $G$-categorical cover for $S$ to a $G$-categorical cover of some invariant neighborhood $U$ of $S$ with the same cardinality. Together with an equivariant cover for $R$, we obtain a $G$-categorical cover for $M$. This proves (\ref{eq2}). 

Assume that $E$ is not empty. We  use   ideas from \cite{colmanhurder,CM2001,hurdertoeben}  to show
\begin{equation}\label{eq3}
\cat_G(R) \leq \cat_G(E)+q
\end{equation}
 Endow $M$ with a $G$-invariant Riemannian metric, then  the projection $\rho$   becomes a Riemannian submersion for an appropriate metric on the open manifold  $B$. Let $\cH \subset TR_0$ denote the orthogonal bundle to the orbits of $G$, so that $\cH$ is $G$-invariant, and is the horizontal distribution for $\rho$  in the sense of Riemannian submersions.

 Given an open set $U\subset B$, set $\whU = \rho^{-1}(U)$. Given a $C^1$-contraction $h:U\times I\to B$  to a point $b_0 \in B$, we define a $G$-equivariant lift $H : \whU  \times I \to R_0$ by requiring that for $x \in   \whU$, 
\begin{equation}\label{eq:horlift}
H_0(x) = x ~, \quad d\rho \left( \frac{d}{dt} H_t(x)  \right) = \frac{d}{dt} h_t(\rho(x)) ~, \quad \frac{d}{dt} H_t(x) \in \cH
\end{equation}
The differential conditions (\ref{eq:horlift}) mean that $H_t(x)$ is the horizontal curve over $h_t(\rho(x))$, and the $G$-invariance of $\cH$ implies that $H_t$ is a $G$-equivariant map for all $t$. As $H_1$ maps $\whU$ into the $G$ orbit over $b_0$,   $\whU$ is a $G$-categorical set in $R_0$.

Now note that $B$ is a connected open manifold of dimension $q$, so there exists a categorical covering   
$\{U_1, \ldots , U_k\}$ with smooth homotopies, for some  $k \leq q$. Their inverse images $\{\whU_1, \ldots , \whU_k\}$ form a $G$-categorical covering for $R_0$.

 It remains to note that   by the Equivariant Borsuk Theorem, we can extend a $G$-categorical cover for $E$ to a $G$-categorical cover of some invariant neighborhood $U$ of $E$ with the same cardinality.
 
 Finally, in the case where $M = R$ and there are no exceptional orbits, then we modify the above proof only with the remark,  that if $M$ is compact, then $M/G$ is compact, hence admits a categorical covering with at most $q+1$ open sets. \end{proof}

This result can be iterated if applied to the singular stratum $S$, using the more general equivariant version of Borsuk's theorem.

\section{Lower bound estimates}\label{sec:lower}

Our next aim is to give two lower bounds for $\cat_G(M)$ which are fundamental in applications. Note that each orbit $Gx$ is a $G$-subspace.

\begin{dfn}
A $G$-{\it path} from an orbit $Gx$ to an orbit $Gy$ is a $G$-equivariant map map $I:Gx\times [0,1]\to M$ such that 
\begin{enumerate}
\item $I_0$ is the inclusion of $Gx$ in $M$,
\item   $I_1(Gx)=Gy$.
\end{enumerate}
 If, in addition, every $I_t$ is a diffeomorphism then we call $I$ a $G$-{\it isotopy}.
 \end{dfn}
Recall that given a $G$-invariant subset $X \subset M$,  a $G$-homotopy is a continuous family of $G$-maps 
$H_t : X \times [0,1] \to M$. For each $x \in X$ we then obtain a $G$-path $H_t : Gx \times [0,1] \to M$ from $Gx$ to $Gy$ where $y = H_1(x)$.

We will now recall some properties of proper actions. Given    a closed subgroup   $H \subset G$, we denote by $(H)$ the conjugacy class of $H$ in $G$. While for an orbit $Gx$ the isotropy group $G_y$ depends on the choice of $y\in Gx$, the conjugacy class $(G_x)$ does not and is therefore an invariant of $Gx$. The class $(G_x)$ is called the {\it orbit type} of $Gx$. There is a partial order on the set of orbit types of the $G$-space $M$: given   isotropy groups $H,K\subset G$, 
\begin{equation}\label{eq5}
(H)\leq (K)\quad \mbox{if}\ gKg^{-1}\subset H\ \mbox{for some}\ g\in G.
\end{equation}

An orbit $Gx$ and its orbit type $(G_x)$ are called {\it principal} if $Gx$   has a $G$-invariant open neighborhood that contains no orbit of larger orbit type. The union of principal orbit types is open and dense in $M$. If $M$ is connected, then the space of principal orbits  is connected,  and hence there is exactly one principal orbit type; this orbit type is comparable to any other orbit type and it is the maximum with respect to the partial order in (\ref{eq5}).
The orbit $Gx$ and its orbit type $(G_x)$ are called {\it minimal} if $(G_x)$ is a minimum with respect to this partial order. 

\begin{lem} 
Let $I:Gx\times [0,1]\to M$ be a $G$-path in $M$, and  write $x_t=I_t(x)$. 
Then for all $0 \leq t \leq 1$,
\begin{equation}\label{equ1}
G_x\subset G_{x_t}\quad \mbox{and therefore}\quad (G_{x_t})\leq (G_x), ~ \mbox{hence} ~ \dim Gx_t\leq \dim Gx.
\end{equation}
\end{lem}
\begin{proof}
For $g\in G_x$ we have $gI(t,x)=I(t,gx)=I(t,x)$.
\end{proof} 
 The second property of (\ref{equ1}) means that $I_t$ respects the partial order of orbit types. It follows that a minimal orbit type is preserved under a $G$-path, i.e. $(G_{x_t})=(G_x)$. We will generalize this principle in the next paragraphs.

For an isotropy group $H$, define the $(H)$-{\it orbit type submanifold}
\begin{equation}
M_{(H)} =\{x\in M\mid G_x\in (H)\}
\end{equation}
which is the union of orbits of the same orbit type $(H)$.  One knows that $M_{(H)}$ is a submanifold, possibly open, and the quotient map  $M_{(H)} \to M_{(H)}/G$ is a fiber bundle when restricted to connected component of $M_{(H)}$.
Also, define the invariant set
\begin{equation}
M_{\leq(H)}=\{x\in M\mid (G_x)\leq (H)\}=G\cdot\Fix(H)
\end{equation}
which  is closed by the Tubular Neighborhood Theorem, but in general need  not be    a submanifold.

Let $H \subset G$ be an isotropy subgroup, and suppose that  $Gx\subset M_{\leq(H)}$ 
(respectively,  $M_{\leq(H)}\cap U\neq\emptyset$). Then   
 equation (\ref{equ1})   implies
\begin{equation}\label{equ2}
I_t(Gx)\subset M_{\leq(H)}\quad \mbox{respectively,}\quad H_t(M_{\leq(H)}\cap U)\subset M_{\leq(H)}.
\end{equation}
Hence
\begin{equation}\label{equ3}
\cat_G(M_{\leq (G_x)})\leq\cat_G(M).
\end{equation}
This proves again that each minimal orbit type is preserved under G-homotopy.

Note that a $G$-homotopy $H_t$ preserves the connected components of $M_{\leq(H)}$. This motivates the introduction of the following: for $x\in M$ we define $\mathcal M_x=G\cdot(M_{(G_x)})_x$, the $G$-orbit of the connected component of $M_{(G_x)}$ containing $x$. Let us call a union of connected components of an invariant set of $M$ a $G$-{\it component} if it is itself invariant. Then  $G\cdot(M_{(G_x)})_x$ is the smallest $G$-component of $M_{(G_x)}$ containing $x$. It is not difficult to show that $\mathcal M_x$ is the union of orbits that can be reached from $Gx$ by a $G$-isotopy. The decomposition $\mathfrak M'$ of $M$ into connected components of orbit type submanifolds $M_{(H)}$ is a Whitney stratification (see e.g. \cite{duistermaat}). The decomposition
$$
\mathfrak M=\{\mathcal M_x\mid x\in M\}.
$$
is a coarser Whitney stratification; the $G$-orbit of one element in $\mathfrak M'$ constitutes one element in $\mathfrak M$. Both stratifications induce the same Whitney stratification of $M/G$. The {\it incidence relations} on $\mathfrak M$ are defined by the partial order 
\begin{equation}\label{defn:ordering}
\mathcal M_y\preceq \mathcal M_x :\Longleftrightarrow  \mathcal M_y\subset \overline{\mathcal M}_x
\end{equation}
which is related to the already introduced partial order on $\mathfrak M'$ by
$$
\mathcal M_y\preceq \mathcal M_x \Longrightarrow (G_y)\leq (G_x), \mbox{i.e.}\ M_{(G_y)}\subset M_{\leq(G_x)}.
$$

The following is a property of a stratification.
\begin{lem}\label{lemma:boundary}
 ${\mathcal M}_y\prec  {\mathcal M}_x \iff y\in\overline{\mathcal{M}}_x\backslash \mathcal M_x\iff \mathcal M_y\subset\overline{\mathcal{M}}_x\backslash \mathcal M_x$.
\end{lem}

Here ${\mathcal M}_y\prec  {\mathcal M}_x$ means ${\mathcal M}_y\preceq  {\mathcal M}_x$, but ${\mathcal M}_y\neq {\mathcal M}_x$. The lemma shows that $\mathcal M_x\preceq \mathcal M_y$ and $\mathcal M_y\preceq \mathcal M_x$ implies $\mathcal M_x=\mathcal M_y$.

\begin{dfn}
An orbit $Gx$ is {\it locally minimal} if $\overline{\mathcal M}_x$ has a $G$-invariant open neighborhood $U$ that contains no smaller orbit type. In this case $\mathcal M_x$ is called a {\it locally minimal stratum}.
\end{dfn}

The above notion will allow us to give a lower bound for $\cat_G(M)$ (see Theorem \ref{thm:lowerbound}) and is illustrated by the examples in section \ref{sec:examples}. It is also surprisingly connected to the question whether the transverse saturated LS-category of a Riemannian foliation is finite or not (see \cite{hurdertoeben}, which also gives a detailed discussion of the properties of locally minimal strata in section 6.) Obviously, a minimal orbit with respect to the orbit type relations is locally minimal. From Lemma~\ref{lemma:boundary}  we derive the following characterization:
\begin{prop}\label{prop:closed}
A stratum $\mathcal M_x$ is locally minimal if and only if it is minimal with respect to the incidence partial order, if and only if $\mathcal M_x$ is closed.
\end{prop}
\begin{proof}
By Lemma \ref{lemma:boundary},  ${\mathcal M}_x$ is closed if and only if it is minimal with respect to the incidence partial order. If ${\mathcal M}_x$ is closed then it is locally minimal by the equivariant Borsuk Theorem. Now assume it is locally minimal. Let $y\in \overline{\mathcal M}_x$. By the Tubular Neighborhood Theorem we have $(G_y)\leq (G_x)$ and, since ${\mathcal M}_x$ is locally minimal,  $(G_y)=(G_x)$. Thus $y\in \overline{\mathcal M}_x$, so ${\mathcal M}_x$ is closed.
\end{proof}

\begin{rem}
For all $x \in M$, $\overline{\mathcal M}_x$ always contains a locally minimal stratum.
\end{rem}

Let $\mathfrak M_{0}$ be the set of locally minimal strata. Let $\mathcal M_x\in\mathfrak M_{0}$ and $V$ a invariant neighborhood in which there is no smaller orbit type than $(G_x)$. Then $\mathcal M_x=M_{\leq (G_x)}\cap V$. 
For a $G$-homotopy $H:U\times [0,1]\to M$, by equation (\ref{equ2}) we have
\begin{equation}\label{equ6}
H_t({\mathcal M}_x\cap U)\subset {\mathcal M}_x, \quad\mbox{hence}\quad \cat_G({\mathcal M}_x)\leq\cat_G(M).
\end{equation}
\begin{thm}\label{thm:lowerbound}
Let $G$ be a Lie group acting properly on $M$. Then 
$$
|\mathfrak M_{0}|\leq\sum_{\mathcal M_x\in\mathfrak M_{0}}\cat_G(\mathcal M_x)\leq \cat_G(M).
$$
\end{thm}

\begin{proof}
We show the second inequality; then the first follows directly. Let  $\{U_j\}_{j\in J}$ be a covering by $G$-categorical sets with corresponding $G$-equivariant  homotopies $H^j:U_j\times [0,1]\to M$. Let $y_j \in M$ be such that $H^j_1(U_j) = G y_j$. For each $\mathcal M_x\in\mathfrak M_{0}$ let $J(\mathcal M_x)\subset J$ denote the subset of indices for which    $U_j \cap \mathcal M_x$ is non-empty. By equation (\ref{equ6}) the restriction of the corresponding homotopies to $\mathcal M_x$ gives a $G$-categorical covering of $\mathcal M_x$; in particular $\cat_G(\mathcal M_x)\leq |J(\mathcal M_x)|$. On the other hand, the image $H^j_1(U_j \cap \mathcal{M}_x)$ is contained in a single orbit $Gy_j$ which must lie in $U_j \cap \mathcal M_x$ by (\ref{equ6}).  Hence, each open set $U_j$ intersects at most one element of $\mathfrak M_0$, so the $J(\mathcal M_x), \mathcal M_x\in\mathfrak M_{0}$ are disjoint. This proves the statement.
\end{proof}

\begin{rem}
Marzantowicz gives in \cite{marzantowicz} an upper bound for $\cat_G(M)$ in terms of minimal orbit types. See Colman \cite{colman} for a refinement in case of finite $G$.
\end{rem}

Each component of the fixed point set of an action is a locally minimal set (each point of it is an orbit with minimal orbit type $(G)$.) Thus we have:

\begin{cor}\label{cor:fix}
Let $G$ be a Lie group acting properly on $M$.
The number of components of the fixed point set is a lower bound for $\cat_G(M)$.
\end{cor}
This corollary also justifies the refinement of orbit types by separating its orbit type submanifolds into its basic components, the $G$-components. Counting the number of minimal orbit type submanifolds would not provide a good lower bound, as the entire fixed point set is the sole orbit type submanifold $M_{(G)}$. 

\begin{ex}\label{ex:SU}
Let us prove that  
\begin{equation}
n+1\leq\cat_{SU(n+1)}SU(n+1)
\end{equation}
for the equivariant category of the action of $G=SU(n)$ on itself by conjugation. The center $Z(G)$ of a Lie group $G$ is the fixed point set $\Fix_G(G)$ of its action on itself by conjugation. In this case we have $Z(SU(n+1))=\{e^{\frac{2\pi ik}{n+1}}I_{n+1}\mid k=0,\ldots,n\}$. By Corollary \ref{cor:fix} we have $n+1\leq\cat_{SU(n+1)}SU(n+1)$. This lower bound is optimal as we will see in Example~\ref{ex:Un}. For the conjugation action of other Lie groups, say $SO(2n)$ for example, this estimate is not optimal and can be improved by the estimate in Theorem \ref{thm:lowerbound}.
\end{ex}

LS category theory has a close relationship with critical point theory for functions, and the next concept develops this parallel for $G$-category. 

\begin{dfn}\label{defn:hierarchy}
A {\it hierarchy} of $(M,G)$ is a monotone function $f:\mathfrak M'\to\mN_0$ with respect to the ordering (\ref{eq5}). That is,  $(G_x)\leq (G_y)$ implies $f(M_{(G_x)})\leq f(M_{(G_y)})$.
\end{dfn}

 A hierarchy $f:\mathfrak M'\to\mN_0$   has  a natural extension  to a map $\bar f:M\to\mN_0$ defined by $\bar f(x):=f(M_{(G_x)})$. Then  $f$ defines a partition $\mathfrak M_f'$ of $M$ by 
 \begin{equation}\label{defn:h-partition}
\mathfrak M_f'=f^{-1}(\mN_0)=\{f^{-1}(n)\mid n\in\mN_0\} \quad \mbox{and set} \quad  \mathfrak M_{f,n}':=\bigcup_{i=0}^n f^{-1}(i)
\end{equation}

We can think of a hierarchy  $f$ as a function on the directed graph of orbit types of $(M,G)$ (see section~2.8 of \cite{duistermaat}) that respects the partial order between the vertices. Note that there is a directed vertex from any orbit type to a principal orbit type.

\begin{ex} \label{ex:hf}
We consider a few examples to illustrate hierarchy functions.
\begin{enumerate}
\item The most common example is given by $f(M_{(G_x)}):=\dim Gx$. This function is well defined by Lemma~\ref{lemma:boundary}, and  $\mathfrak M_f'$ is the stratification of $M$ by orbit dimension. Let $S^k=\mathfrak M_{f,k}'$.  \\

\item We can refine the last example. The action of $G$ on $S_k:=S^k\backslash S^{k-1}$ gives a foliation by orbits of dimension $k$. For $x \in S_k$ its   orbit   $Gx \subset S_k$ may have      normal holonomy, as a leaf of the induced foliation. This holonomy group must be finite. Moreover, if $G$ is compact then  there are at most finitely many orbits with holonomy, and we can define 
 $n_k$ as the maximal cardinality of the holonomy groups of leaves in $S_k$. For $G$ non-compact, we assume that each $n_k$ is finite. 
Let $S_{k,i}$ for $i=1,\ldots,n_k$ be the union of orbits with holonomy group of cardinality $i$. 
  Now define 
 \begin{equation}
\bar f(x)=i-1+\sum_{j=1}^kn_j ~ \mbox{ if }  ~ x\in S_{k,i}
\end{equation}
This map only depends on the orbit type  and therefore defines    a hierarchy, with $S_f'$ being  the   refined stratification of $M$. 
This refines the   stratification of $M$ by dimension, to a stratification by dimension and holonomy cardinality. The stratification $S_f'$  has been introduced in the study of the singular Riemannian foliations defined by the leaf closures in Riemannian foliations (see e.g. \cite{Haefliger1985, Haefliger1988,molino}.) 
\\

\item The following hierarchy is defined if there is only one principal orbit type $(H)$ (e.g. if $M/G$ is connected) and there are only a finite number of orbit types (e.g. if $G$ is compact). Let $(G_x)$ be an arbitrary orbit type. A chain
$$
\mathcal C:\quad (G_x)=(H_1)\leq \cdots \leq (H_n)=(H)
$$
of orbit types from $(G_x)$ to $(H)$ is said to have length $n$. Let $l(G_x)$ be the maximal length of chains from $(G_x)$ to $(H)$ and $L=\max\{l(G_x)\mid x\in M\}$. We define $\bar f(x):=L-l(G_x)$. \\

\item Alternatively we can define $\bar f(x)$ as the maximal length of chains from some minimal orbit type to $(G_x)$. 
This yields the stratification by holonomy groups used in the paper \cite{colmanhurder} on the study of transverse LS-category for compact Hausdorff foliations.

\end{enumerate}
\end{ex}

A hierarchy function provides a lower bound estimate on the $G$-category: 
\begin{prop}
Let $G$ be a Lie group acting properly on $M$. 
Let $f:\mathfrak M'\to\mN_0$  be a hierarchy function, with induced map 
 $\bar f:M\to\mN_0$. Then   
$\bar f$ is lower semicontinuous; that is,  
$$\liminf_{y\to x}\bar f(y)\geq\bar f(x).$$
Hence the strata  $\mathfrak M_{f,n}'$ are preserved under $G$-homotopy, and  thus 
\begin{equation}
\cat_G(\mathfrak M_{f,n}')\leq \cat_G(\mathfrak M_{f,n+1}')\leq\cat_G(M)
\end{equation}
\end{prop}

\begin{proof}
Let $x_n$ be a sequence in $M$ converging to $x$. By the Tubular Neighborhood Theorem we have $(G_x)\leq (G_{x_n})$ for large $n$. Since $f$ is monotone,  $\bar f(x)\leq \bar f(x_n)$ for large $n$. Thus $\bar f$ is lower semicontinuous. We now show that $\mathfrak M_{f,n}'$ is preserved under $G$-homotopy. Let $H:U\times [0,1]\to M$ be a $G$-homotopy with $x\in U$. Then $(G_{x_t})\leq (G_x)$, so $\bar f(x_t)\leq \bar f(x)$, i.e. $x_t\in \mathfrak M_{f,f(x)}'$.
\end{proof}
We apply this result  to the hierarchy $f$ of example (\ref{ex:hf}.2) to obtain: 

\begin{cor}\label{cor:hierlowerbound}
Let $G$ be a Lie group acting properly on $M$. Then we have
\begin{equation}\label{eq:superest}
\cat_G(S^0)\leq\cat_G(S_{k,i})\leq\cat_G(S_{l,j})\leq \cat_G(S^l)\leq\cat_G(S)\leq \cat_G(M), 
\end{equation}
for $k\leq l$ and if $k=l$ for $i\geq j$. Here, $S$ denotes the singular stratum.
\end{cor}

\section{Polar actions}\label{sec:polar}

In this section we will review the definition and properties of a polar action. In section~\ref{sec:upper}, we will give an upper bound for the category of a polar action in terms of the action of its   Weyl group.

\begin{dfn}\label{def:polar}
Let $G$ Lie group acting smoothly by isometries on a complete Riemannian manifold $M$. A \emph{section} for the $G$-action is an isometrically immersed complete  submanifold $i:\Sigma\to M$ which meets every orbit and always orthogonally. Then the  dimension of $\Sigma$ is equal to   the   cohomogeneity of the action, which was denoted by $q$.   Note that for any $g \in G$, the map $g \circ  i \colon g\Sigma \to M$  is   again a section. 
A {\it polar action} is a $G$-action with a section. If   $\Sigma$ is a flat submanifold, then the action is called {\it hyperpolar}.
\end{dfn}

\begin{rem}
The set of regular points in a section is open and dense in it. A section is always a totally geodesic submanifold (see \cite{szenthe}). 
\end{rem}

\begin{rem} The immersion $i$ might not be injective. If injectivity fails in a regular point of the action, we can write $i=j\circ\rho$ where $\rho:\Sigma\to\Sigma'$ is a covering map and $j:\Sigma'\to M$ is a section that is injective in regular points. We will always assume that $i$ is reduced in this sense. On the other hand, injectivity can still fail at singular points of the action. 
\end{rem}

The geometry of polar actions has been extensively studied 
\cite{berndttamaru,dadok,kollross1,kollross2,kostant,palaisterng,podestathorbergsson,szenthe,thorbergsson1,thorbergsson2}. 
Let us consider a few examples. 

\begin{ex}\label{ex:hyperpolar} The following examples are all hyperpolar actions.
\begin{enumerate}
\item Isometric cohomogeneity one actions. The sections are the normal geodesics of a regular orbit. These have been classified in special cases, although remains an open problem to classify all such actions \cite{kollross1, berndttamaru}.\\

\item A compact Lie group $G$ with bi-invariant metric acting on itself by conjugation. The maximal tori are the sections. \\

\item  Let $N$ be a symmetric space. The identity component of the isometry group,  $G=I(N)_0$, acts transitively on $N$.  We can write $N=G/K$, where $K=G_p$ for some point $p\in N$, and $(G,K)$ is called a {\it symmetric pair}. Then the isotropy action 
$$K\times G/K\to G/K~ ; \quad (k,gK)\mapsto kgK$$ 
and its linearization $K\times (T_{[K]}G/K)\to T_{[K]}G/K$ at the the tangent space to the point [K],    are hyperpolar. The sections are the maximal flat submanifolds through $[K]$, and their tangent spaces in $[K]$, respectively. These are called   \emph{s-representations}.\\

\item Let $(G,K_1)$ and $(G,K_2)$ be two symmetric spaces of the above form. Then the left action of $K_1$ on $G/K_2$ and its linearization are hyperpolar. These actions are called {\it Hermann actions}. They generalize examples (2) and (3).
\end{enumerate}
\end{ex}
\begin{rem}
Dadok has classified all linear representations that are polar in \cite{dadok}: they are orbit equivalent to the linearized actions of example (3), the s-representations. Kollross has classified in \cite{kollross1} all hyperpolar actions on irreducible, simply-connected symmetric spaces of compact type: they are of type (1) and of type (4). Polar actions that are not hyperpolar on symmetric spaces of compact type have been found only on compact rank one symmetric spaces; for a classification, see \cite{podestathorbergsson}. For a survey on these objects as well as on polar actions, see \cite{thorbergsson1} and \cite{thorbergsson2}.
\end{rem}

\begin{thm}[Slice Theorem for polar actions \cite{palaisterng}]
Let $G$ a Lie group with a proper polar action on $M$. Then the slice representation of $G_x$ on  $\nu_x (Gx)$ is hyperpolar with sections of the form $T_x\Sigma$, where $\Sigma$ is a section through $x$.
\end{thm}

Note that any polar action acts transitively on the set of sections. Indeed, let $Gx$ be an arbitrary regular orbit. By definition any two sections $\Sigma_1$ and $\Sigma_2$ meet $Gx$ say in $x_1$ respectively $x_2$. Let $g\in G$ such that with $gx_1=x_2$. Then $g_*\nu_{x_1}Gx=\nu_{x_2}Gx$. Since $\nu_{x_i}Gx=T_{x_i}\Sigma_i$ and $\Sigma_i$ is totally geodesic, we have $g(\Sigma_1)=\Sigma_2$. We have the following corollary.
\begin{cor}\label{cor:transitive}
The  isotropy group $G_x$ acts transitively on the sections through $x$.
\end{cor}

Let us also note that  Dadok's classification of hyperpolar actions   implies:
\begin{cor}\label{dadok}
The slice representation of a proper polar action is an s-re\-pre\-sen\-ta\-tion.
\end{cor}

\section{Weyl group and $G$-equivariant blow-up}\label{sec:weyl}

We recall the definition of the Weyl group for a polar action, which generalizes the Weyl group of a classical Lie group $G$, acting on itself via the adjoint map.  
\begin{dfn} 
Let $G$ a Lie group acting smoothly by isometries on a complete Riemannian manifold $M$, and assume the action is polar with section  $i:\Sigma\to M$. Let 
\begin{eqnarray}
N := N_G(\Sigma) & = & \{g\in G\mid g(i(\Sigma))=i(\Sigma)\}\\
Z := Z_G(\Sigma) & = & \{g\in G\mid gi(x)=i(x)\ \mbox{for any}\ x\in\Sigma\}
\end{eqnarray}
Then the Weyl group is 
 \begin{equation}
W=N_G(\Sigma)/Z_G(\Sigma)
\end{equation}
\end{dfn}
The action of $N$ descends to an action on $\Sigma$ for which $i:\Sigma \to M$ is $N$-equivariant. The generalized Weyl group $W$ therefore acts effectively and isometrically on $\Sigma$. The orbits of $G$ and $W$ are related in the following way:
\begin{equation}\label{eqn:orbits}
Gi(x)\cap i(\Sigma)=i(Wx)
\end{equation}
or in a more suggestive form by $Gx\cap\Sigma=Wx$ for any $x\in\Sigma$.
Note that the normalizer $N_G(\Sigma)$ of the section $\Sigma$ need not be discrete;  in example (\ref{ex:hyperpolar}.2) the section $\Sigma$ is a maximal torus $T$, and $T=Z_G(T)\subset N_G(\Sigma)$.  

\begin{prop}
The action of the Weyl group on $\Sigma$ is properly discontinous.
\end{prop}
\begin{proof}
Since $G$ acts properly on $M$, the orbits are closed and embedded. In particular,  every $G$-orbit intersects $\Sigma$ discretely. By (\ref{eqn:orbits}) the $W$-orbits must also be discrete. For an isometric action, this is   equivalent to the action being  properly discontinuous, i.e. for any compact subset $K$ of $\Sigma$, the set of $w\in W$ with $w(K)\cap K\neq\emptyset$ is finite.
\end{proof}

We next introduce  the ``blow-up'' of a polar action. 
 This blow-up is different from the inductive normal projectivization of strata beginning with the lowest dimensional (see e.g. \cite{duistermaat}). The blow-up we use here has been introduced in the context of singular Riemannian foliations admitting sections in \cite{boualem} and studied further in \cite{toeben1} (see also \cite{toeben2}, \cite{toeben3}.)

Recall that for   $x \in M$ a regular point,   the orbit $Gx$ has maximal dimension, hence there exists   exactly one section $\Sigma_x$ containing $x$: 
since the intersection of $\Sigma_x$ with $Gx$ is orthogonal and $\Sigma_x$ is totally geodesic,  we have $\Sigma_x=\exp^{\perp}(\nu_x(Gx))$. 

For a singular point $y \in M$,  there is a family of sections running through this point. The blow-up of a singular point $y$   is obtained using the space of all sections    through the singular point.

Let  $G_q(TM)$ denote the Grassmann bundle  of $q$-planes in $TM$.   The fiber  $G_q(T_xM)$ over $x \in M$  is the Grassmann manifold of $q$-planes in $T_xM$. Recall $r$ is the dimension of a regular orbit, so that $G_q(TM) \cong Fr(TM)/(O(r) \times O(q))$ where $Fr(TM) \to M$ denotes the orthogonal frame bundle of $TM$ with the right action by $O(r + q)$.

Given a section $i : \Sigma \to M$, let  $\tau_x \Sigma  = i_*(T_x\Sigma)$ denote  the tangent space $T_x\Sigma$  considered  as a  subspace of $T_xM$ via the map $i_*$ and hence as 
an element of the Grassmannian $G_q(T_x\Sigma)$.

Define the (Grassmann)  {\it blow-up} of $(M,G)$ by
\begin{equation}\label{eq:blowup}
\whM :=\{\tau_x \Sigma = i_*(T_x\Sigma)\mid i:\Sigma\to M\ \mbox{is a section}, x\in\Sigma\}\subset  G_q(TM)
\end{equation}

Let $\whpi: \whM \to M$ be the restriction of the canonical projection $\pi : G_q(TM)\to M$ to $ \whM $; so $\whpi(\tau_x \Sigma))=i(x)$. 
The action of $G$ on $ \whM $ defined by 
$$g_* \tau_x\Sigma =  (g \circ i)_* (T_x \Sigma)$$
 is proper and  the projection $\whpi$ is $G$-equivariant with respect to this action.
The blow-up $ \whM $ can be endowed with a differentiable structure such that its inclusion $\iota: \whM \
\hookrightarrow G_q(TM)$ is an immersion (see section~3.2 of \cite{toeben1}.) 

Note that given a section  $i : \Sigma \to M$ there is a tautological  lift to a map $\tau : \Sigma \to  \whM $, where for $x \in \Sigma$ we set $\tau (x) = \tau_x\Sigma$.

\begin{prop}\label{prop:product}
Let $G$ be a Lie group with a proper polar action on    $M$,  $i:\Sigma\to M$   a section, and 
  $N=N_G(\Sigma)$   the normalizer of  $\Sigma$.
Then there is a $G$-equivariant diffeomorphism $ \whM \cong G\times_N \Sigma$.
\end{prop}
\begin{proof} 
The action of the normalizer group $N$   on $G\times \Sigma$ is defined by  $h \cdot (g,x)=(gh^{-1},hx)$.
The space $G\times_W\Sigma$ is the quotient by this action, which is just the quotient space defined by the equivalence relation  $(gh,x) \sim (g,hx)$ for $h \in N$.  

Consider the map $\tilde{\Phi}(g,x) = g_*\tau_x\Sigma$. Then 
$$\tilde{\Phi}(gh,x) = (gh)_*\tau_x\Sigma = g_*(h_*\tau_x\Sigma) = g_* \tau_{hx}\Sigma = \tilde{\Phi}(g,hx)$$

where we use that $h\Sigma = \Sigma$ as $h \in N_G(\Sigma)$. Thus, there is a well-defined map  
\begin{eqnarray}
\Phi \colon G\times_N \Sigma&\to & \whM  \label{eq:weyliso}\\
\Phi(\mbox{[}(g,x)\mbox{]})& = & g_* \tau_x\Sigma \nonumber
\end{eqnarray}

The Lie group $G$ acts on $G\times_N \Sigma$ by $g\cdot [(h,x)]=[(gh,x)]$. 

First, let us show that   $\Phi$ is onto. Every orbit of $G$ in $M$ intersects the image of $\Sigma$, so it suffices to show that given $x \in \Sigma$, the action of the isotropy group $G_x$  is transitive on sections through $x$. This follows from Corollary~\ref{cor:transitive}.

We claim that $\Phi$ is injective. Assume $[(g,x)]$ and $[(h,y)]$ have the same image, where $x,y \in \Sigma$. Without loss we can assume that $h=e$. 
Then $g_* \tau_x\Sigma=\tau_y\Sigma$ which implies $gx = y$ and $g(\Sigma)=\Sigma$, i.e. $g\in N$.
Therefore $[(g,x)]=[(e,y)]$.
\end{proof}

\begin{ex}\label{ex:better} 
 Let $G$ be a connected Lie group and $K \subset G$ a   compact subgroup. Assume the center $Z(G)$ is not empty, and $\Gamma \subset Z(G)$ is a finite subgroup which acts effectively,  isometrically on a compact manifold $N$. Now assume that the quotient space    $M = G/K \times_{\Gamma} N$ is a manifold, where for each $h \in \Gamma$ and $y \in N$, we identify $(gK h, y) \sim (gK, hy)$. Let $[(gH,y)] \in M$ denote the equivalence class of $(gK, y)$. 
 Then  $i : N \to M$, $i(y) = [(eK, y)]$, is a section and 
 $N_G(\Sigma) = K\Gamma$, $Z_G(\Sigma) = K$, so $W = \Gamma/(\Gamma \cap K)$. The diffeomorphism  of 
 Proposition~\ref{prop:product} is  the tautology 
 $G \times_{K\Gamma} N \cong G/K \times_{\Gamma} N$.

This class of examples  are standard models in the theory of compact Hausdorff foliations \cite{Millett1974}, where 
the quotient manifold $B  = G\backslash M \cong W\backslash N$  is    an orbifold. Conversely, given an orbifold $B$ of dimension $q$, there is an associated manifold $M$ with a locally-free action of $O(q)$, so that $B \cong O(q) \backslash M$. However, for such a group action, there need not exists a section $\Sigma$.

  \end{ex}
 
 \begin{ex}\label{ex:torus}  (Polar coordinates on ${\mathbb R}^3$) 
 Let $M = {\mathbb S}^2 \subset {\mathbb R}^3$ be the unit sphere, and let $G = SO(2)$ act via rotations in the plane $(x,y,0)$. Let $\Sigma$ be the embedding $i \colon {\mathbb S}^1 \to {\mathbb S}^2$ given by $i(w) = (\sin(\theta), 0, \cos(\theta))$ for $w = (\sin(\theta), \cos(\theta) \in  {\mathbb S}^1$. Then $N_G(\Sigma) = \{\pm 1\} \subset SO(2)$, $Z_G(\Sigma) = \{1\}$, $W = {\mathbb Z}/2{\mathbb Z}$ and 
 Proposition~\ref{prop:product} yields 
  $$ G \times \Sigma =  {\mathbb S}^1 \times  {\mathbb S}^1  \longrightarrow  \whM =  {\mathbb S}^1 \times_{{\mathbb Z}/2{\mathbb Z}} {\mathbb S}^1 \longrightarrow  M =  {\mathbb S}^2$$ 
  which is just  the standard blow-down map of the 2-torus to obtain the 2-sphere.
\end{ex}

\section{Category for  polar actions}\label{sec:upper}
We can now give an upper estimate for the equivariant category of a proper polar action in terms of the proper action of the  Weyl group  on a section $\Sigma$. The category of the   Weyl group action is  often easier to compute, as the action is discrete.

\begin{thm}\label{thm:main}
Let $G$ be a Lie group with a proper polar action on    $M$,  $i:\Sigma\to M$   a section, and 
  $W=N_G(\Sigma)/Z_G(\Sigma)$   the generalized Weyl group acting on $\Sigma$. Then 
\begin{equation}\label{eq:weylcat}
\cat_G(M)\leq\cat_W(\Sigma)
\end{equation}
\end{thm}
\begin{proof}
Let $\mathcal{U}=\{U_i\}_{i\in I}$ be a $W$-categorical covering of $\Sigma$,  with $W$-equivariant homotopies $H_i:U_i\times [0,1]\to \Sigma$. Let $V_i=G\cdot U_i$ be the orbit saturation of $U_i$ for the $G$-action on $M$. Then $\whV_i:=\whpi^{-1}(V_i)\subset  \whM $ is the orbit saturation of the $G$-action on the blow-up $ \whM $. 
The strategy of the    proof is to show that the   sets $\whV_i$ are $G$-categorical in $ \whM $, and hence the $V_i$ form  a $G$-categorical cover for $M$.

Note that $U_i$ is $W$-invariant by assumption, so  invariant under the induced action of   $N=N_G(\Sigma)$. Using the identification $ \whM \cong G\times_N \Sigma$, 
each  $\whV_i$  thus  has the alternate description
\begin{equation}
\whV_i=\bigcup_{[g]\in G/N}g_*\tau (U_i),
\end{equation}
Note that if  $g_*\tau (U_i) \cap h_*\tau(U_i) \ne \emptyset$ then $h^{-1}g \in N$, or $g = hk$ for some $k \in N$, as $U_i \subset \Sigma$ always contains regular points for the $G$-action. Thus the above union is disjoint.

Define $(\whH_i)_t|g_*\tau (U_i) = (g_*\circ \tau) \circ H_t\circ (g_*\circ \tau)^{-1}$. We claim this yields a well-defined map on $\whV_i$. Recall that the homotopy $H_i$ is assumed to be $W$-equivariant, and that $g_* \circ \tau = \tau \circ g$ for all $g \in G$. 

Let   $y, y' \in   U_i$ with $g_*\tau (y) = h_* \tau (y')$ so that  $g = hk$ for some $k \in N$. Then     
$y = g^{-1}h y' = k^{-1}y'$ and we calculate
\begin{eqnarray*}
(\whH_i)_t(g_* \tau(y)) & = & (g_*\circ \tau) \circ H_t\circ (g_*\circ \tau)^{-1}(g_*\tau(y)) \\
& = &  (g_*  \circ \tau) \circ H_t(y)  \\
& = &  (g_*\circ \tau) \circ H_t (k^{-1} y')  \\
 & = &  (g_*\circ \tau) \circ k^{-1}   H_t  (y')  \\
 & = &  (g_*k^{-1}_*\circ \tau) \circ    H_t  (y')  \\
 & = &  (h_*\circ \tau) \circ    H_t  \circ (h_*\circ \tau)^{-1}(h_*\tau(y')) \\
 & = &  (\whH_i)_t(h_*\tau(y')) 
\end{eqnarray*}

Thus, $(\whH_i)_t : \whV_i \to  \whM $ is well-defined, and $G$-equivariant by construction.

We next show that each $\whV_i$ is $G$-categorical in $ \whM $;  that is, the image 
  $(\whH_i)_1(\whV_i)$ is contained in a $G$-orbit in $ \whM $. By the definitions, the following diagram commutes
$$
\begin{array}{rcccl}
&g_*\tau U_i&\stackrel{(\whH_i)_t}{\longrightarrow}&g_*\tau \Sigma&\\
g_*\circ\tau&\uparrow&&\uparrow&g_*\circ\tau\\
&U_i&\stackrel{(H_i)_t}{\longrightarrow}&\Sigma&
\end{array}
$$
For $t =1$,  the map $(H_i)_1 : U_i \to W y_i$ for some $y_i \in \Sigma$. Thus, the image  
$(\whH_i)_1(\whV_i)$ is contained in the union of the images of the sets 
$g_* \circ \tau(W y_i)$ for $g \in G$, which are all contained in the $G$-orbit 
$G \cdot \tau(y_i)$.  We have thus  shown that $\cat_G( \whM )\leq \cat_W(\Sigma)$.

Now we want to define a homotopy $H'_i:V_i\times [0,1]\to M$ satisfying 
\begin{equation}
\whpi\circ (\whH_i)_t=(H_i')_t\circ\whpi. 
\end{equation}
The above formula determines $H'_i$ if it is well-defined, because $\whpi:\whV_i\to V_i$ is surjective. 
Also the $G$-equivariance of $(H_i')_t$  follows immediately using $g\circ\whpi=\whpi\circ g_*$. 

For existence, it suffices to show that $(\whH_i)_t$ respects $\whpi$-fibers; that is, 
\begin{equation}
\sigma ,  \sigma' \in \whV_i ~, ~ \whpi(\sigma) = \whpi(\sigma') ~ \Longrightarrow ~ \whpi\circ (\whH_i)_t(\sigma) = \whpi\circ (\whH_i)_t(\sigma').
\end{equation}
Let  $\sigma , \sigma' \in \whV_i$   satisfy  $\whpi(\sigma) = \whpi(\sigma')$. 
Then $\sigma = g_*\tau_x\Sigma$ for some  $g\in G$ and $x\in U_i$. Moreover, 
$\whpi(\sigma) = gx$ and so $\sigma'$ is a section through $gx$.
By Corollary~\ref{cor:transitive} $G_{\whpi(\sigma)}$ acts transitively on the set of sections through $\whpi(\sigma)$ which is $\whpi^{-1}(\whpi(\sigma))$, so $\whpi^{-1}(\whpi(\sigma))=(G_{\whpi(\sigma)})_*\sigma $. Hence there exists $h \in G_{gx}$ such that 
$$\sigma' = h_* \sigma = h_*  g_*\tau_x\Sigma =  (hg)_*\tau_x\Sigma. $$

Set $x_t = (H_i)_t(x)$. We show in Proposition~\ref{prop:fixset} below that $G_x \subset G_{x_t}$ for all $0 \leq t \leq 1$, hence
$$G_{gx}=gG_xg^{-1}\subset gG_{x_t}g^{-1}=G_{gx_t}$$
and so 
$$\whpi\circ (\whH_i)_t(\sigma') =  \whpi\circ (\whH_i)_t((hg)_* \tau_x \Sigma) = hg x_t = g x_t = \whpi\circ (\whH_i)_t(\sigma)$$
as was to be shown. 
\end{proof}

It remains to show that the isotropy groups are stable under the homotopies $(H_i)_t$.
We use the same notation as above.

\begin{prop}\label{prop:fixset}
Given $U \subset \Sigma$, 
let $H:U\times [0,1]\to \Sigma$ be a $W$-equivariant homotopy. For $x\in U$ we write $x_t =H(t,x)$. Then $G_{x}\subset G_{x_t}$ for any $t\in [0,1]$.
\end{prop}
\begin{proof} 
Clearly $W_{x}\subset W_{x_t}$. This means $x_t\in(\Fix W_x)_x$, which is the connected component of $\Fix W_x$ containing $x$. 

Recall that the action of $W = N_G(\Sigma)/Z_G(\Sigma)$ on $\Sigma$ is the quotient of the action of $N_G(\Sigma) \subset G$ on $\Sigma \subset M$. We are going to show 
$$
(\Fix W_x)_x\subset(\Fix G_x)_x\cap \Sigma.
$$ 
This implies $x_t\in \Fix G_x$, so $G_x\subset G_{x_t}$.

We first want to show $\Fix (dW_x|T_x\Sigma)=\Fix (dG_x|V)\cap T_x\Sigma$, where $V=\nu_x(Gx)$. We remark that the right hand side is equal to $\Fix (dG_x|V)$ as $\Fix dG_x|V\subset T_x\Sigma$; this follows from the fact that $G_x$ and its identity component $G_x^0$ act transitively on the set of sections through $x$. The linear action of $G_x^0$ on $\nu_x (Gx)$ is hyperpolar (and therefore orbit-equivalent to an s-representation) for which $T_x\Sigma$ is a section. Let $W^0$ be the associated Weyl group acting on $T_x\Sigma$; it is generated by the reflections through the singular hyperplanes in $T_x\Sigma$ through the origin. It is known that $W^0=N_{G_x^0}(\Sigma)/Z_{G_x^0}(\Sigma)$. Let $F_1=\Fix (W^0)$ and $F_2$ be its orthogonal complement in $T_x\Sigma$. This decomposes $T_x\Sigma$ into two $W^0$-invariant subspaces. Let $\pi:V\to T_x\Sigma$ be the orthogonal projection and let $D=\ker\pi$. As the image of an orbit $dG_x^0v, v\in V$ under $\pi$ is the convex hull of $W^0\pi(v)$ by Kostant's convexity theorem (\cite{kostant}, see also \cite{palaisterng} Theorem 8.6.2 and 8.6.4), $F_2\oplus D$ is a $dG_x^0|V$-invariant subspace and therefore also $F_1$. The action of $dG_x^0$ on $F_1$ is trivial (for $v\in F_1$, the orbit $dG_x^0v$ lies on the sphere of radius $\|v\|$ and on the other hand on $\pi^{-1}(v)$, since again by convexity $\pi(dG_x^0v)=W^0v=\{v\}$; the intersection of both submanifolds in $V$ is exactly $\{v\}$), so $F_1\subset \Fix (dG_x^0|V)\cap T_x\Sigma$. The converse is obviously true. 
Thus $\Fix (W^0)=\Fix (dG_x^0|V)\cap T_x\Sigma$.\par 
Now we want to show $(\Fix W_x)_x=(\Fix G_x)_x\cap \Sigma$. The natural inclusion
$$
\phi:N_{G}(\Sigma)_x/N_{G_x^0}(\Sigma)\to G_x/G_x^0
$$
is an isomorphism. We prove surjectivity. Let $[g]\in G_x/G_x^0$. Since $G_x^0$ acts transitively on the set of sections through $x$ there is an $h \in G_x^0$ with $(gh)_*(T_x\Sigma)=T_x\Sigma$. Thus $gh\in N_{G}(\Sigma)_x$ and this proves surjectivity. We now prove injectivity. Assume $\phi([n_1])=\phi([n_2])$ for $n_i\in N_{G}(\Sigma)_x$. Thus $n_1h=n_2$ for some $h\in G_x^0$. It follows $h_*T_x\Sigma=T_x\Sigma$. Therefore $h\in N_{G_x^0}(\Sigma)$ and we have proven injectivity.\par
Thus $G_x/G_x^0=\{n_iG_x^0\}_{i\in I}$ for a countable index set $I$ and $n_i\in N_{G}(\Sigma)_x$. Together with $\Fix (W^0)=\Fix (dG_x^0|V)$ this implies $\Fix (dW_x|T_x\Sigma)=\Fix (dG_x|V)\cap T_x\Sigma=\Fix (dG_x|T_xM)\cap T_x\Sigma$ and therefore by exponentiating we obtain $(\Fix W_x)_x=(\Fix G_x)_x\cap\Sigma$. 
\end{proof}

\section{Examples and applications}\label{sec:examples}

We provide a selection of examples of homogeneous  polar actions which show that the upper and lower bounds for the $G$-category provided in   Theorem~\ref{thm:main}, and Theorem~\ref{thm:lowerbound} and Corollary~\ref{cor:hierlowerbound}, can be very effective. In particular, we deduce the calculation of the the classical Lusternik-Schnirelmann category for $U(n)$ and $SU(n)$
 due to Singhof \cite{singhof} mentioned in the Introduction.

\subsection{The LS-category and the equivariant category of $SU(n)$ and $U(n)$}\label{ex:Un}
We will first prove 
$$
\cat_{SU(n+1)}SU(n+1)=n+1
$$ 
for the equivariant category of the action of $G=SU(n)$ on itself by conjugation. We have already seen in example \ref{ex:SU} that $n+1$ is a lower bound.\par
We now want show with the help of Theorem \ref{thm:main} that $n+1$ is also an upper bound.
The maximal torus of this action is
$$
\mT^n=\{\lambda_1\oplus\cdots\oplus\lambda_{n+1}\mid \lambda_i\in S^1\subset\mC,\lambda_1\cdots\lambda_{n+1}=1\}.
$$
The Weyl group $W_{SU(n+1)}$ is the group of permutations of the coordinates of $\mT^n$
$$
\sigma:\lambda_1\oplus\cdots\oplus\lambda_{n+1}\mapsto \lambda_{\sigma(1)}\oplus\cdots\oplus\lambda_{\sigma(n+1)}.
$$
Let $z_k=e^{\frac{2\pi ik}{n+1}}I_{n+1}, k=0,\ldots,n$ be the set of central elements of $SU(n+1)$. We give   a $W$-categorical covering $U_k$ of $\Sigma$ so that each $U_k$ contracts radially to $z_k$.

Let $\varphi:\mR^{n+1}\to \mT^{n+1}, x=(x_0,\ldots,x_n)\mapsto (e^{2\pi i x_0},\ldots,e^{2\pi i x_n})$ be the canonical covering map. The preimage of $\mT^n \subset \mT^{n+1}$ under $\varphi$ intersected with the fundamental domain $[0,1]^{n+1}$ is exactly $\{x\in [0,1]^{n+1}\mid \sum_{i=1}^n x_i=k, k\in \{1,2,\ldots,n\}\}$.

Define the $n$-simplex $\Delta_k=\{x\in [0,1]^{n+1}\mid \sum_{i=1}^n x_i=k\}$ for $k=1,\ldots,n$. We observe that $\varphi$ restricted to the interior $\intern(\Delta_k)$ of $\Delta_k$ in $\{x\in \mR^{n+1}\mid \sum_{i=0}^nx_i=k\}$ is a diffeomorphism onto its image in $\mT^n$. 
Clearly $\intern(\Delta_k)$ is invariant under permutation of coordinates and its radial contraction to $(\frac{k}{n},\ldots, \frac{k}{n})$ is equivariant with respect to the permutation group. The conjugation of this homotopy with $\varphi$ is $W$-equivariant and contracts to $z_k$. 
An extension of this homotopy to a neighborhood of $\varphi(\Delta_k)$ would finish the proof but this is not possible.
The injectivity of $\varphi$ on the entire $\Delta_k$ fails exactly in its vertices which are mapped to $z_0$. 

Now let $V_k'$ be a small open neighborhood of $\Delta_k$ in $\{x\in \mR^{n+1}\mid \sum_{i=0}^nx_i=k\}$ invariant under permutation of coordinates minus an $\epsilon/2$ ball around the vertices of $\Delta_k$ for small $\epsilon>0$ such that $V_k'$ is still star-shaped with respect to $(\frac{k}{n},\ldots, \frac{k}{n})$. Then, for appropriate choices, $\varphi$ is injective and therefore an isometry of $V_k'$ to $V_k$. 

The $\{V_k\}, k=1,\ldots,n$ cover $\mT^n$, together with the ball $V_0$ around $z_0$ of radius $\epsilon$. Each $V_k$ is $W$-categorical and radially  contracts to $z_k$ via the $W$-equivariant homotopy $h_k$.

Now the proof of Theorem \ref{thm:main} shows that $h_k$ can be extended to a $SU(n+1)$-homotopy $H_k$ of $SU(n+1)\cdot V_k\subset SU(n+1)$ to $SU(n+1)\cdot z_k=z_k$. This gives us a $SU(n+1)$-categorical covering of $SU(n+1)$ of cardinality $n+1$. Thus $\cat_{SU(n+1)}SU(n+1)=n+1.$

We can now quickly reprove the following theorem by Singhof  \cite{singhof}.
\begin{thm}[Singhof] 
The LS-categories of the unitary and the special unitary groups are $\cat(SU(n))=n$ and $\cat(U(n))=n+1$.
\end{thm} 
Since the $SU(n+1)$-homotopy $H_i$ from above contracts to a point $z_k=SU(n+1)\cdot z_k$, the open set $SU(n+1)\cdot V_k$ is also categorical in the classical sense of Lusternik and Schnirelmann. Therefore our equivariant cover also provides a LS-categorical cover of $SU(n+1)$. So $\cat(SU(n+1))\leq n+1$. On the other hand $n+1\leq\cat(SU(n+1))$ by the general formula $\mbox{cuplength}(M)+1\leq\cat(M)$ and $\mbox{cuplength}(SU(n+1))=n$. Now $\cat(U(n))=n+1$ follows from $U(n)\cong S^1\times SU(n)$, from the formula $\cat(M\times N)+1\leq\cat(M)+\cat(N)$ and $\mbox{cuplength}(U(n))=n$ .

\subsection{The LS-category $\mCP^n$}
We want to give an alternative computation of the LS-category of $\mCP^n$ with the help of Theorem \ref{thm:main} as in the previous example. Consider the action of $\mT^n=\{c=(c_0,\ldots,c_n)\in \mT^{n+1}\mid c_0\cdots c_n=1\}$ on an element $z=[z_0:\ldots: z_n]\in\mCP^n$ by $c\cdot z:= [c_0z_0:\ldots c_nz_n]$. This action is polar. The natural embedding of $\mRP^n$ into $\mCP^n$ is a section. The elements of the Weyl group $W$ are the even sign changes in the homogeneous coordinates. The fixed points of $W$ are the $n+1$ points $e_i=[0:\ldots:0:1:0:\ldots:0]$, where the $1$ is at position $i$, for $i=0,\ldots,n$. These are also $\mT^n$-fixed points. For each $i$ we want to define a $W$-equivariant homotopy that contracts to $e_i$. 

Define $\varphi_i:\mR^n\to \mRP^n; x=(x_1\ldots x_n) \mapsto [x_1:\ldots:1: \ldots: x_n]$. Let $U_i \subset \mRP^n = \Sigma$ denote the image of $\varphi_i$.
The radial contraction to the origin in $\mR^n$ gives us via $\varphi_i$ a $W$-homotopy $h^i: U_i\times [0,1]\to \mRP^n; ([z],t)\mapsto [t\frac{z_0}{z_i}:\ldots:1:\ldots :t\frac{z_n}{z_i}]$ contracting to $e_i$. The collection $\{U_0, \ldots, U_n\}$ form  a $W$-categorical cover of the section $\mRP^n$. (Of course, the open sets $U_i$  are just the usual covering of $\mRP^n$ by Grassmann cells.)

Now set $V_i = \mT^n\cdot U_i$
By the proof of Theorem \ref{thm:main}, we can extend the $h_i$ to $\mT^n$-homotopies $H_i: V_i \times [0,1]\to \mCP^n$. These contract to the $\mT^n$-fixed points $e_i$. 
Together with the lower bound from the cuplength we now have $\cat(\mCP^n)=n+1$. 

Note that together with the lower bound from Corollary \ref{cor:fix}, the arguments also show $\cat_{\mT^n}(\mCP^n)=n+1$.

 \eject

\subsection{The LS-category $\mHP^n$}
We consider the polar action of $G=Sp(1)\cdots Sp(1)$ ($n$ factors) on $\mHP^n$. It has the same section, namely $\mRP^n$, and Weyl group $W$ as the previous action. The $W$-fixed points are also $G$-fixed points. This gives us as before $\cat(\mHP^n)=n+1$.

\subsection{The LS-category $\mOP^2$}
We consider the polar action of $G=Spin(8)$ on $\mOP^n$ with section $\mOP^2$ with Weyl group as before for $n=2$. Then $\cat(\mOP^2)=3$.

\bigskip

\end{document}